\documentclass[a4paper,10pt]{article}
\usepackage[latin1]{inputenc}
\usepackage[T1]{fontenc}
\usepackage{amssymb}
\usepackage{epsfig}
\usepackage{amssymb}
\usepackage{amsmath}
\usepackage{amsfonts}
\usepackage{graphicx, url}
\usepackage{makeidx}
\usepackage{cite}
\usepackage{float}
\usepackage{fancyhdr}
\usepackage{color}

\newtheorem{theorem}{Theorem}[section]

\newtheorem{remark}[theorem]{Remark}
\newtheorem{assumption}[theorem]{Assumption}

\newtheorem{example}[theorem]{Example}

\newenvironment {proof} {{\it Proof.}}{\hspace*{\fill}$\Box$\par\vspace{4mm}}
\usepackage{amsmath}
\usepackage{amsfonts}

\newcommand{\mc}{\mathcal}
\newcommand{\mb}{\mathbb}
\newcommand{\bb}{\textbf}
\newcommand{\cl}{\mbox{\rm cl}}

\newcommand{\co}{\mbox{\rm co}}

\begin{document}

\title{Management of a hydropower system via convex duality}

\author{Kristina Rognlien Dahl \footnote{Department of Mathematics, University of Oslo. kristd@math.uio.no.}}

\maketitle

\abstract{
We consider the problem of managing a hydroelectric power plant system. The system consists of $N$ hydropower dams, which all have some maximum production capacity. The inflow to the system is some stochastic process, representing the precipitation to each dam. The manager can control how much water to release from each dam at each time. She would like to choose this in a way which maximizes the total revenue from the initial time $0$ to some terminal time $T$. The total revenue of the hydropower dam system depends on the price of electricity, which is also a stochastic process. The manager must take this price process into account when controlling the draining process. However, we assume that the manager only has partial information of how the price process is formed. She can observe the price, but not the underlying processes determining it. By using the conjugate duality framework of Rockafellar~\cite{Rockafellar}, we derive a dual problem to the problem of the manager. This dual problem turns out to be simple to solve in the case where the price process is a martingale or submartingale with respect to the filtration modelling the information of the dam manager.  
}

\section{Introduction}
\label{sec: intro}

The framework of this paper is inspired by those presented in Huseby~\cite{Huseby} and Alais et al.~\cite{Alais}, as both consider management of hydropower dams. We consider the problem of managing a hydroelectric power plant system. The system consists of $N$ hydropower dams, which all have some maximum production capacity. The inflow to the system is some stochastic process, representing the precipitation or other natural source of inflow such as snow melting or streams to each dam. The manager of the facility can choose how much water to turbine from each dam at each time. She would like to choose this in a way which maximizes the total revenue from the initial time $0$ to some terminal time $T$. The total revenue of the hydropower dam system depends on the price of electricity, which is assumed to be a stochastic process. The manager must take this price process into account when controlling the draining process. Hence, the dilemma of our dam manager is how much water she should drain at each time, when she must respect the natural constraints of the facility (e.g. the maximum storage capacity and maximum production capacity) as well as take into consideration the uncertainty regarding inflow and electricity price. We apply the conjugate duality framework of Rockafellar~\cite{Rockafellar} to derive a dual problem to the initial problem of the dam manager. This dual problem turns out to be simple to solve in the case where the price process is a martingale or submartingale with respect to the filtration modelling the manager's information. For a brief introduction to conjugate duality theory, see the Appendix~\ref{sec: conjugate}.

In the paper by Huseby~\cite{Huseby}, the allocation of draining between the different dams is the focus. Hence, they do not maximize the total income, but instead they aim to satisfy the demand for electricity. In addition, the techniques they use to solve the problem is completely different from ours. 

In Alais et al.~\cite{Alais}, they consider a single multi-usage hydropower dam. The goal is to maximize the expected gain under a bound on the control, non-anticipativity of the draining strategy and a tourist constraint. This constraint requires that the water level of the dam is high enough during the tourist season with a certain probability. This is a different kind of constraint than what we have. 

Chen and Forsythe~\cite{Chen} consider problem somewhat similar to ours, but in continuous time. They derive a Hamilton Jacobi Bellman equation, which turns out to be a partial integrodifferential equation and study and use viscosity solutions to study its properties.

In contrast, we consider a discrete time, arbitrary scenario space setting and use conjugate duality techniques to derive a dual problem. To the best of our knowledge, such methods have not been applied to hydroelectric dam management problems before. However, in mathematical finance, the use of duality methods have been extensively studied by for Pennanen~\cite{Pennanen1}-\cite{Pennanen2}, Pennanen and Perkki{o}~\cite{Pennanen}, King~\cite{King}, King and Korf~\cite{KingKorf} and others over the past decade. Some advantages with duality methods are:

\begin{itemize}
\item{The optimal value of the dual problem gives a bound on the optimal value of the primal problem.}
\item{In some cases, so-called strong duality holds: The optimal primal value is actually equal to the optimal dual value.}
\item{The method is extremely suitable for handling various kinds of constraints without added complexity. This is not the case for classical stochastic control methods such as stochastic dynamic programming and the stochastic maximum principle, see Ji and Zhou~\cite{JiZhou} for more on this.}
\item{If the case where the problem and the constraints are linear, in the finite scenario space case, the method reduces to linear programming which can be solved efficiently using e.g. simplex or interior point methods. Using the simplex algorithm for solving the dual problem at the same time provides the primal optimal control variables as shadow prices of the dual constraints, see e.g. Vanderbei~\cite{Vanderbei}}
\end{itemize}

The structure of the paper is as follows: In Section~\ref{sec: model}, we present the model for the hydropower dam system and the stochastic processes involved. We also introduce the maximization problem of the dam manager, and rewrite this to a more tractable form. Then, in Section~\ref{sec: dual}, we choose a suitable perturbation space and derive its dual space. Based on this, we derive the dual problem. In a special case, this dual problem turns out to be simple to solve, so we find the dual solution in this case. This optimal dual value gives an upper bound on the optimal value of the primal problem. In Section~\ref{sec: strong_dual} we discuss whether strong duality holds for our problem. We use this to discuss some computational properties of the problem, in particular when it is simpler to solve the dual problem than the primal. In Section~\ref{sec: system}, we add structure to the dam system. However, despite the added complexity, the conjugate duality methods works in the same way as before. Finally, in Appendix~\ref{sec: conjugate}, we give an overview of conjugate duality theory for the convenience of the reader.

\section{Modelling a system of hydropower dams}
\label{sec: model}

We consider a system consisting of $N$ hydropower dams over a discrete time period $t=0, 1, 2, \ldots, T < \infty$. This framework is equipped with an (arbitrary) probability space $(\Omega, P, \mc{F})$. Related to this probability space, we have several different stochastic vector processes. For each of the following processes, vector component $i \in \{1, 2,\ldots, N\}$ corresponds to dam $i$ of the facility:

The electricity price process denoted by 
\[
\bb{S}(t,\omega) = (S_1(t,\omega), S_2(t,\omega), \ldots, S_N(t,\omega)) \in \mb{R}^N.
\] 

We interpret $S_i(t,\omega)$ as the price of electricity for dam $i$ at time $t$ if scenario $\omega \in \Omega$ is realized. See Remark~\ref{remark: price} for an explaination of why we consider (potentially) different prices for different dams. Note that we do not make any assumptions on the structure of this process, so it can be any discrete time stochastic process. In the following, we will usually (for ease of notation) omit writing out the $\omega \in \Omega$ in the notation of the various processes. In the following, we will usually (for ease of notation) omit writing out the $\omega \in \Omega$ in the notation of the various processes.
  
The inflow process is $\textbf{R}(t,\omega) \in \mb{R}^N$, and for any time $t$, the random variable $\textbf{R}(t)$ is interpreted as the amount of precipitation and other natural inflow to the dams between times $t$ and $t+1$. The inflow is measured in terms of units of electricity which the water corresponds to. Note that this can take both positive and negative values. In practice, this means that we allow for both positive inflow such as rain and snow melting as well as negative ``inflow'' such as evaporation or natural draining of water.

The amount of water in the dams is denoted by $\textbf{V}(t,\omega) \in \mb{R}^N$. For any time $t$, the random variable $\textbf{V}(t)$ is the amount of water in the dam (measured in terms of units of electricity which the water corresponds to) at time $t$. Note that we must have $\bb{V}(t, \omega) \geq 0$ (since the dams cannot hold a negative amount of water). In particular, the initial water level $\bb{V}(1) \geq 0$.

The draining process is denoted by $\textbf{D}(t,\omega) \in \mb{R}^N$, and for any time $t$, the random variable $\textbf{D}(t)$ is interpreted as the amount of water which is drained from the dams between times $t$ and $t+1$. The draining process is also measured in terms of units of electricity which the water corresponds to. This process can be controlled by the dam manager. When the dam manager chooses $\textbf{D}(t)$, she does so based on her current information, which may only be partial. In particular, in the case where the manager has full information, she observes $\textbf{S}(t), \textbf{V}(t)$ and $\textbf{R}(t-1)$ as well as all previous values of these processes. When selling the electricity from the drained water, the manager will get the unknown price $\textbf{S}(t+1)$. Note also that the manager must choose how much to drain in a period before knowing how large the inflow will be over the same period of time.

An illustration of the order which information is revealed and choices are made is shown in Figure~\ref{fig: info}.

 \begin{figure}[ht]
 \setlength{\unitlength}{0.7mm}
 \begin{picture}(60,40)(-30,0)
  \put(1,4){\line(35,0){120}}

  \put(20,4){\circle*{1}}
  \put(60,4){\circle*{1}}
  \put(100,4){\circle*{1}}
	
  \put(20,8){\makebox(0,0){$\textbf{V}(1)$}}
  \put(20,15){\makebox(0,0){$\textbf{S}(1)$}}
  
  \put(60,8){\makebox(0,0){$\textbf{V}(2)$}}
  \put(60,15){\makebox(0,0){$\textbf{S}(2)$}}
  \put(60,22){\makebox(0,0){$\textbf{R}(1)$}}

	\put(100,8){\makebox(0,0){$\textbf{V}(3)$}}
  \put(100,15){\makebox(0,0){$\textbf{S}(3)$}}
  \put(100,22){\makebox(0,0){$\textbf{R}(2)$}}
  

	\put(40,0){\makebox(0,0){$\textbf{D}(1)$}}
	\put(80,0){\makebox(0,0){$\textbf{D}(2)$}}

  \put(20,1){\makebox(0,0){$t = 1$}}
  \put(60,1){\makebox(0,0){$t = 2$}}
  \put(100,1){\makebox(0,0){$t=3$}}
  \put(115,1){\makebox(0,0){$\ldots$}}
 \end{picture}
\caption{The order of information.}
\label{fig: info}
\end{figure}
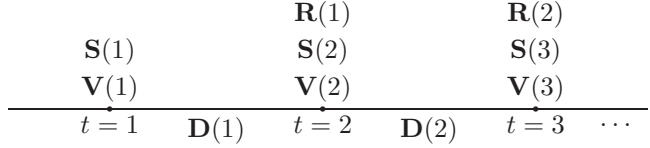

\begin{remark}
{\rm There is an ambiguity in the way we have chosen to interpret the draining process. As an alternative, one could say that the draining happens instantaneously, and hence $\textbf{D}(t)$, the water drained at time $t$ could be sold at the price $\textbf{S}(t)$. This kind of interpretation would eliminate some of the uncertainty of the dam manager. However, not that the manager still has to take into consideration that the water she drains now may have been better off being saved for a later time when the electricity price is potentially higher. Hence, the electricity price is an important source of uncertainty in this case as well.}
\end{remark}

\begin{remark}
\label{remark: price}
{\rm Note that we seemingly consider $N$ different prices of electricity, since $\bb{S}(t) \in \mb{R}^N$. This essentially means that we consider different prices for the different dams. The reason for doing this is that the different prices can be used to reflect different technologies in the dams, different locations of water relative to the turbines etc. Hence, this can be useful even though the actual market price of electricity is just one price. If the dams are equal, one can just let the price vector $\bb{S}(t) = (S(t), S(t), \ldots, S(t))$, where $S(t)$ is the market price of electricity.}
\end{remark}

An overview of the notation is shown in Table~\ref{table: notasjon}.

\begin{figure}
\begin{center}
\begin{tabular}{ l l }
 $\bb{S}(t)$ & Electricity price  \\ 
 $\bb{R}(t)$ & Inflow process \\
 $\bb{V}(t)$ & Amount of water in dams \\
 $\bb{D}(t)$ & Draining process (control variable) \\
 $\bb{m}$ & Maximum amount of water in dams \\
 $\bb{b}$ & Maximal production capacities
 
\end{tabular}
\end{center}
\caption{An overview of the notation.}
\label{table: notasjon}
\end{figure}

Since water cannot be turbined infinitely fast, we also have a maximal production capacity for the different dams. These maximal production capacities are the components of the vector $\bb{b}=(b_1, b_2, \ldots, b_N)$. In addition, the dams have a finite capacity to hold water without flooding, so we let the vector $\bb{m}= (m_1, m_2, \ldots, m_N)$ be the maximal amount of water in the dams. See Figure~\ref{fig: Ndams} for an illustration of the dam system.

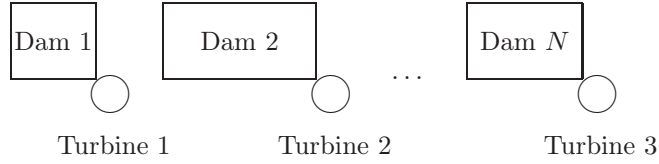
\begin{figure}
\center
\setlength{\unitlength}{1cm}
\begin{picture}(8,4)
\put(0,2){\framebox(1.1,1){Dam $1$}}
\put(2,2){\framebox(2,1){Dam $2$}}
\put(5,2){\mbox{\ldots}}
\put(6,2){\framebox(1.5,1){Dam $N$}}
\put(1.3,1.8){\circle{0.5}}
\put(4.2,1.8){\circle{0.5}}
\put(7.7,1.8){\circle{0.5}}
\put(0.6,1){\mbox{Turbine $1$}}
\put(3.5,1){\mbox{Turbine $2$}}
\put(7,1){\mbox{Turbine $3$}}
\end{picture}
\caption{The hydropower system: $N$ separate dams.}
\label{fig: Ndams}
\end{figure}

The information of the dam manager is given by a filtration $(\mc{G}_t)_{t=1}^T$ which can be a subfiltration of the full information filtration, i.e., the manager may only have partial information. Recall that the full information filtration, denoted $(\mc{F}_t)_{t=1}^T$ is the one generated by $\textbf{S}(t), \textbf{V}(t)$ and $\textbf{R}(t-1)$. Hence, $\mc{G}_t \subseteq \mc{F}_t$ for all times $t=1, 2, \ldots, T$. We assume that $\mc{F}_T = \mc{G}_T = \Omega$, i.e., that at the terminal time the true scenario is revealed to the manager. Note that this concept of partial information is quite general. For instance, the dam manager may have delayed price information, incomplete information about the inflow to the dams or the price formation. In Example~\ref{ex: partial} we consider such a situation.

Let 

\[
\mc{D}_{\mc{G}} := \{\mbox{all } \{\mc{G}_t\}\mbox{-adapted processes}  \}.
\]

The problem of the dam manager is as follows:

\begin{equation}
\label{eq: problem}
\begin{array}{rlll}
\max_{\{\bb{D(t)}\}} &&E[\sum_{t=1}^{T-1} \bb{D}(t) \cdot \bb{S}(t+1) + \alpha \bb{V}(T) \cdot \bb{S}(T)] \\[\smallskipamount]
\mbox{such that}\\[\smallskipamount]
\bb{V}(t+1) &=& \bb{V}(t) + \bb{R}(t) - \bb{D}(t), \quad t=1, \ldots, T-1 \mbox{ a.s.}, \\[\smallskipamount]
\bb{D}(t)  &\leq& \bb{V}(t) \leq \bb{m}, \quad t=1, \ldots,T-1 \mbox{ a.s.}, \\[\smallskipamount]
0 &\leq& D_i(t) \leq b_i, \quad i=1, \ldots, N \mbox{ and } t=1, \ldots, T-1 \mbox{ a.s.}
\end{array}
\end{equation}
\noindent where we maximize over all drain processes $\bb{D} \in \mc{D}_{\mc{G}}$, the initial water level $\bb{V}(1)$ is given and $\alpha$ is a constant. Note that the constant $\alpha$ gives weight to the terminal value of the water in the reservoir. Note that it can be any value, but it is natural to have $\alpha \in (0,1]$. This is natural because the remaining water should have some value, but perhaps not $\bb{S}(T) \cdot \bb{V}(T)$, since it cannot be turbined immediately.

That is, she wants to maximize the total revenue from the hydropower dams while not draining more water than what is available at any time (given the development of the water level) while also respecting the maximum production capacities and maximum water levels of the different dams and not draining more water than what's available at each time.

Note that problem~\eqref{eq: problem} is an infinite linear programming problem, i.e., the problem is linear with infinitely many constraints and variables. For more on infinite programming, see for instance Anderson and Nash~\cite{AndersonNash} and for a numerical method, see e.g. Devolder et al.~\cite{DevolderEtAl}.  However, if $\Omega$ is finite, \eqref{eq: problem} is a linear programming problem. In this case, the problem can be solved numerically using the simplex algorithm or an interior point method, see for example Vanderbei~\cite{Vanderbei}

\bigskip

\begin{example}
\label{ex: partial}
This example is a twist on the one in Dahl~\cite{Dahl}.

In this example, we illustrate a kind of partial information which is not delayed information. Although the results of this paper hold when $\Omega$ is an arbitrary set, we consider a situation where $\Omega$ is finite. This simplifies the intuition and allows for illustration via scenario trees. For computational purposes and practical applications, this is also the most relevant.

Consider times $t=1,2,3$, $\Omega := \{\omega_1, \omega_2, \ldots, \omega_5\}$ and a hydropower facility with only one dam. The inflow process to the dam is $R(t,\omega)$, and the electricity price process is $S(t,\omega)$. Let
\[
R(t,\omega) := X(t,\omega) + \xi(t,\omega),
\]
\noindent i.e.,  the inflow process is composed of two other processes, $X$ and $\xi$. For instance, $X(t)$ may be the precipitation, while $\xi(t)$ is the inflow due to snow-melting.

The seller does not observe these two processes, only the current inflow. The following scenario trees show the development of the processes $X$ and $\xi$, as well as the inflow process observed by the seller.

%
 \begin{figure}[H]
 \setlength{\unitlength}{0.7mm}
 \begin{picture}(50,80)(-40,0)
  \put(20,45){\circle*{3}} 
  \put(40,65){\circle*{3}} 
  \put(40,25){\circle*{3}} 
  \put(60,13){\circle*{3}} 
  \put(60,37){\circle*{3}} 
  \put(60,57){\circle*{3}} 
  \put(60,65){\circle*{3}} 
  \put(60,73){\circle*{3}} 
  \put(20,45){\line(1,1){20}}
  \put(20,45){\line(1,-1){20}}
  \put(40,65){\line(5,0){20}}
  \put(40,65){\line(5,2){20}}
  \put(40,65){\line(5,-2){20}}
  \put(40,25){\line(5,3){20}}
  \put(40,25){\line(5,-3){20}}

  \put(2,53){\makebox(0,0){$\Omega = \{\omega_1, \omega_2, \ldots, \omega_5\}$}}
  \put(37,75){\makebox(0,0){$\xi=3, \{\omega_1, \omega_3, \omega_5\}$}}
  \put(37,12){\makebox(0,0){$\xi=5, \{\omega_2, \omega_4\}$}}
  \put(68,73){\makebox(0,0){$\omega_1$}}
  \put(68,57){\makebox(0,0){$\omega_5$}}
    \put(68,65){\makebox(0,0){$\omega_3$}}
  \put(68,37){\makebox(0,0){$\omega_2$}}
  \put(68,13){\makebox(0,0){$\omega_4$}}

  \put(1,4){\line(35,0){80}}
  \put(20,4){\circle*{1}}
  \put(40,4){\circle*{1}}
  \put(60,4){\circle*{1}}

  \put(20,1){\makebox(0,0){$t = 0$}}
  \put(40,1){\makebox(0,0){$t = 1$}}
  \put(60,1){\makebox(0,0){$t = T = 2$}}
 \end{picture}
\caption{The process $\xi$}
\end{figure}
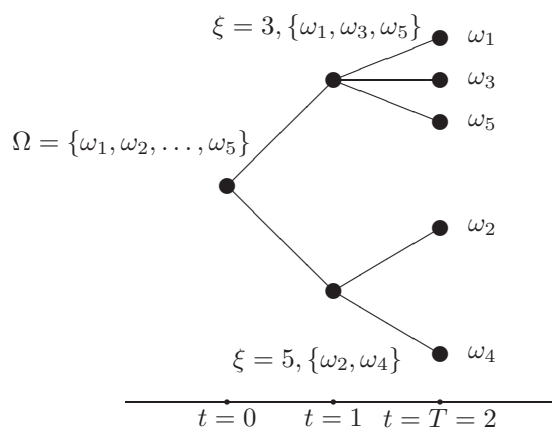

\bigskip

%
 \begin{figure}[H]
 \setlength{\unitlength}{0.7mm}
 \begin{picture}(50,80)(-40,0)
  \put(20,45){\circle*{3}} 
  \put(40,65){\circle*{3}} 
  \put(40,25){\circle*{3}} 
  \put(60,13){\circle*{3}} 
  \put(60,25){\circle*{3}} 
  \put(60,37){\circle*{3}} 
  \put(60,57){\circle*{3}} 
  \put(60,73){\circle*{3}} 
  \put(20,45){\line(1,1){20}}
  \put(20,45){\line(1,-1){20}}
  \put(40,65){\line(5,2){20}}
  \put(40,65){\line(5,-2){20}}
  \put(40,25){\line(5,3){20}}
  \put(40,25){\line(5,0){20}}
  \put(40,25){\line(5,-3){20}}

  \put(2,53){\makebox(0,0){$\Omega = \{\omega_1, \omega_2, \ldots, \omega_5\}$}}
  \put(37,75){\makebox(0,0){$X=4, \{\omega_1, \omega_2\}$}}
  \put(37,12){\makebox(0,0){$X=2, \{\omega_3, \omega_4, \omega_5\}$}}
  \put(68,73){\makebox(0,0){$\omega_1$}}
  \put(68,57){\makebox(0,0){$\omega_2$}}
  \put(68,37){\makebox(0,0){$\omega_3$}}
  \put(68,25){\makebox(0,0){$\omega_4$}}
  \put(68,13){\makebox(0,0){$\omega_5$}}

  \put(1,4){\line(35,0){80}}
  \put(20,4){\circle*{1}}
  \put(40,4){\circle*{1}}
  \put(60,4){\circle*{1}}

  \put(20,1){\makebox(0,0){$t = 0$}}
  \put(40,1){\makebox(0,0){$t = 1$}}
  \put(60,1){\makebox(0,0){$t = T = 2$}}
 \end{picture}
\caption{The process $X$}
\end{figure}
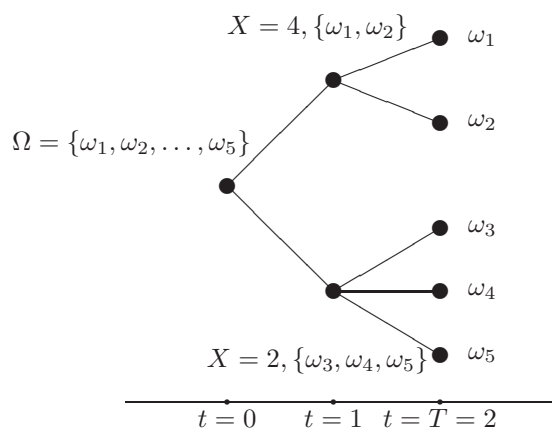

\bigskip

Full information in this market corresponds to observing both processes $X$ and $\xi$ (as well as full information corresponding to the price process $S(t)$), i.e., the full information filtration (w.r.t. the inflow) $(\mc{F}_t)_t$ is the sigma algebra generated by $X$ and $\xi$, $\sigma(X,\xi)$. However, the filtration observed by the seller $(\mc{G}_t)_t$, generated by the inflow process $R(t)$, is (strictly) smaller than the full information filtration. For instance, if you observe that $\xi(1)=3$ and $X(1)=4$, you know that the realized scenario is $\omega_1$. However, this is not possible to determine only through observation of the inflow process $R(t)$. Hence, this is an example of a model with hidden processes, which is a kind of partial information that is not delayed information.

%
 \begin{figure}[H]
 \setlength{\unitlength}{0.7mm}
 \begin{picture}(50,80)(-40,0)
  \put(20,45){\circle*{3}} 
  \put(40,65){\circle*{3}} 
  \put(40,25){\circle*{3}} 
   \put(40,45){\circle*{3}} 
 \put(60,13){\circle*{3}} 
  \put(60,32){\circle*{3}} 
  \put(60,57){\circle*{3}} 
  \put(60,73){\circle*{3}} 
 \put(60,45){\circle*{3}} 
  \put(20,45){\line(1,1){20}}
  \put(20,45){\line(1,-1){20}}
   \put(20,45){\line(1,0){20}}
 \put(40,65){\line(5,2){20}}
  \put(40,65){\line(5,-2){20}}
  \put(40,25){\line(3,1){20}}
  \put(40,45){\line(1,0){20}}
  \put(40,25){\line(5,-3){20}}

  \put(-5,53){\makebox(0,0){$R=6, \Omega = \{\omega_1, \omega_2, \ldots, \omega_5\}$}}
  \put(40,75){\makebox(0,0){$R=7, \{\omega_1, \omega_4\}$}}
  \put(40,12){\makebox(0,0){$R=9, \{\omega_3, \omega_5\}$}}
   \put(40,40){\makebox(0,0){$R=5, \{\omega_2\}$}}
  \put(73,73){\makebox(0,0){$R= 3, \omega_1$}}
  \put(73,56){\makebox(0,0){$R=8, \omega_4$}}
    \put(73,44){\makebox(0,0){$R=9, \omega_2$}}
  \put(73,32){\makebox(0,0){$R=7, \omega_3$}}
  \put(73,13){\makebox(0,0){$R=4, \omega_5$}}

  \put(1,4){\line(35,0){80}}
  \put(20,4){\circle*{1}}
  \put(40,4){\circle*{1}}
  \put(60,4){\circle*{1}}

  \put(20,1){\makebox(0,0){$t = 0$}}
  \put(40,1){\makebox(0,0){$t = 1$}}
  \put(60,1){\makebox(0,0){$t = T = 2$}}
 \end{picture}
\caption{The inflow process $R(t)$}
\end{figure}
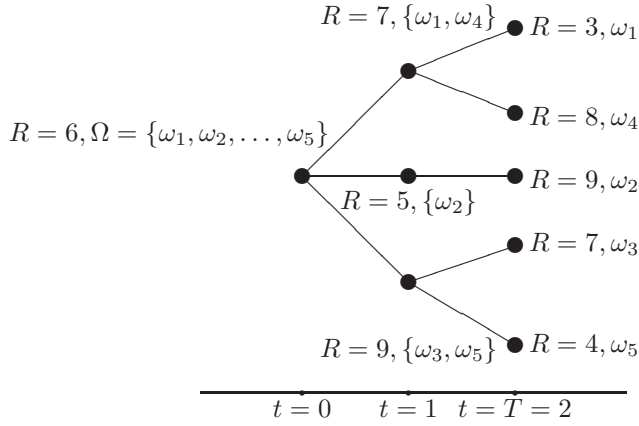
\end{example}

It turns out that we can rewrite the problem in such a way that we remove the $\bb{V}(t)$ process. Since $\bb{V}(t+1) = \bb{V}(t) + \bb{R}(t) - \bb{D}(t), t=1, \ldots, T-1$, we have:

\[
\begin{array}{lll}
 \Delta \bb{V}(t) &:=& \bb{V}(t+1) - \bb{V}(t) \\
&=& \bb{R}(t)-\bb{D}(t). 
\end{array}
\]

Hence,

\begin{equation}
\label{eq: mellomregning}
 \begin{array}{lll}
  \bb{V}(t) &=& \bb{V}(t) - \bb{V}(t-1) + \bb{V}(t-1) - \bb{V}(t-2) + \bb{V}(t-2) - \ldots \\[\smallskipamount]
  &&- \bb{V}(2) + \bb{V}(2) - \bb{V}(1) + \bb{V}(1) \\[\smallskipamount]
  &=& \sum_{s=1}^{t-1} \Delta \bb{V}(s) + \bb{V}(1) \\[\smallskipamount]
  &=& \sum_{s=1}^{t-1} (\bb{R}(s)- \bb{D}(s)) + \bb{V}(1).
 \end{array}
\end{equation}

Therefore, problem~\eqref{eq: problem} can be rewritten:

\[
\begin{array}{lll}
\label{eq: problem_omskrevet}
\max_{\bb{D}} \hspace{0.05cm} E[\sum_{t=1}^{T-1} \bb{D}(t) \cdot \bb{S}(t+1) + \alpha \bb{S}(T) \cdot \Big(\sum_{s=1}^{T-1} (\bb{R}(s)- \bb{D}(s)) + \bb{V}(1)\Big) ] \\[\smallskipamount]
\mbox{such that}
\end{array}
\]
\begin{equation}
\begin{array}{rlll}
\bb{D}(t)  &\leq& \sum_{s=1}^{t-1} (\bb{R}(s)- \bb{D}(s)) + \bb{V}(1) \leq \bb{m}, t=1, \ldots,T-1 \mbox{ a.s.}, \\[\smallskipamount]
0 &\leq& D_i(t) \leq b_i, i=1, \ldots, N \mbox{ and } t=1, \ldots, T-1 \mbox{ a.s.}
\end{array}
\end{equation}
\noindent where $\sum_{s=1}^0 \ldots = 0$, so $\bb{D}(1) \leq \bb{V}(1) \leq \bb{m}$. Also, as previously, the maximization is over all $\bb{D} \in \mc{D}_{\mc{G}}$.

\begin{remark}
{\rm In the current framework, we consider $N$ different hydro-dams, but since they are not connected to one another and there is a separate maximal production capacity for each dam, we would not lose anything by considering just one dam instead. However, in Section~\ref{sec: system}, we will add network structure connecting the dams. To have a consistent notation, we choose to formulate the problem in vector form from the beginning. 

Note also that if the hydropower facility has a maximal total production capacity $C < N(m_1 + m_2 + \ldots + m_N)$, we could not just consider one dam instead of $N$ dams.} 
\end{remark}

\section{The dual problem}
\label{sec: dual}

In this section, we will use the conjugate duality framework of Rockafellar~\cite{Rockafellar} to derive a dual problem to the rewritten version of the manager's problem~\eqref{eq: problem_omskrevet}. See the Appendix~\ref{sec: conjugate} for an overview of this theory.

\begin{remark}
 \label{remark: conjugate_duality}
{\rm The main idea of conjugate duality is to represent the original, or \emph{primal}, problem as one half of a minimax problem where a saddle point exists. The other half of this minimax problem is called the \emph{dual problem}. In order to do this, we introduce a function $K(\bb{D},\bb{y})$, called \emph{the Lagrange function}, depending on some perturbation variables $\bb{y}$ such that there exists a saddle point for this function. The function $K(\bb{D},\bb{y})$ is chosen such that our primal problem is $\sup_{\bb{D} \in \mc{D}_{\mc{G}}} \inf_{\bb{y} \in Y} K(\bb{D},\bb{y})$. Then, the dual problem is $\inf_{\bb{y} \in Y} \sup_{\bb{D} \in \mc{D}_{\mc{G}}} K(\bb{D},\bb{y})$, and this optimal value bounds our primal problem (from above). Under some conditions, these optimal primal and dual values coincide (and are attained in the saddle point of $K(\bb{D},\bb{y})$). In this case we say that \emph{strong duality} holds. For a more detailed presentation of conjugate duality theory, see the Appendix~\ref{sec: conjugate} and Rockafellar~\cite{Rockafellar}.}
\end{remark}

Let $p \in [1, \infty)$, and define the perturbation space 

\[
U = \{u \in L^p(\Omega, \mc{F}, P: \mb{R}^{4(T-1)N} | \quad \bb{u}= (\bb{u}_{\gamma}, \bb{u}_v, \bb{u}_{\lambda}, \bb{u}_w) \}
\]
\noindent where writing $\bb{u}= (\bb{u}_{\gamma}, \bb{u}_v, \bb{u}_{\lambda}, \bb{u}_w)$, $\bb{u}_i(\omega) \in \mb{R}^{(T-1)N}$ for $i=\gamma, v, \lambda, w$, is to highlight the different parts of the permutation vector corresponding to the constraints of the manager's problem~\eqref{eq: problem_omskrevet}.

Corresponding to the perturbation space $U$, we define the dual space 

\[
Y = U^* = \{\bb{y} \in L^q(\Omega, \mc{F}, P: \mb{R}^{4(T-1)N} | \quad \bb{y}= (\gamma, \bb{v}, \lambda, \bb{w}) \}
\]
\noindent where the vector of dual variables is $\bb{y} = (\gamma, \bb{v}, \lambda, \bb{w})$. Here, $\gamma$ is the vector of $\gamma_{i,t}, i=1, \ldots, N, t=1, \ldots, T-1$ and similarly for $\bb{v}$,$\lambda$ and $\bb{w}$. Note that the components of $\bb{y}$ and $\bb{u}$ correspond to one another.

We define a bilinear pairing between the dual spaces $U$ and $Y$ by $\langle \bb{u}, \bb{y} \rangle = E[\bb{u} \cdot \bb{y}]$, where $\langle \cdot, \cdot \rangle$ denotes the Euclidean inner product. The perturbation function $F : \mc{D}_{\mc{G}} \times U \rightarrow \mb{R}$ is defined in the following way:

\medskip

\noindent $F(\bb{D},\bb{u}) = E[\sum_{t=1}^{T-1} \bb{D}(t) \cdot \bb{S}(t+1) + \alpha \bb{S}(T) \cdot \Big(\sum_{s=1}^{T-1} (\bb{R}(s)- \bb{D}(s)) + \bb{V}(1)\Big) ]$ 

if 
\begin{equation}
\label{eq: perturb}
\begin{array}{lll}
\bb{D}(t) - \bb{b} &\leq& \bb{u}_v, \\[\smallskipamount]
-\bb{D}(t) &\leq& \bb{u}_{\gamma}, \\[\smallskipamount]
\bb{D}(t) - \sum_{s=1}^{t-1} (\bb{R}(s)- \bb{D}(s)) - \bb{V}(1) &\leq& \bb{u}_{\lambda, t} \quad \mbox{ for } t=1, \ldots, T-1, \\[\smallskipamount]
\sum_{s=1}^{t-1} (\bb{R}(s)- \bb{D}(s)) + \bb{V}(1) -\bb{m} &\leq& \bb{u}_{w, t} \quad \mbox{ for } t=1, \ldots, T-1
\end{array}
\end{equation}
and $F(\bb{D},\bb{u}) =-\infty$ otherwise.

Then, the Lagrange function is

\[
\begin{array}{lll}
 K(\bb{D},\bb{y}) &=  E[\sum_{t=1}^{T-1} \bb{D}(t) \cdot \bb{S}(t+1) + \alpha \bb{S}(T) \cdot \Big( \sum_{t=1}^{T-1}(\bb{R}(t)- \bb{D}(t)) + \bb{V}(1)\Big)] \\[\smallskipamount]
 & + E[\sum_{t=1}^{T-1} \gamma \cdot \bb{D}(t)] + E[\sum_{t=1}^{T-1} \bb{v}_{t} \cdot (\bb{b} - \bb{D}(t))] \\[\smallskipamount]
 &+ E[\sum_{t=1}^{T-1} \lambda_{t} \cdot (\bb{V}(1) + \sum_{s=1}^{t-1} (\bb{R}(s) - \bb{D}(s)) - \bb{D}(t))] \\[\smallskipamount]
 & + E[\sum_{t=1}^{T-1} \bb{w}_{t} \cdot \big(\bb{m} - \bb{V}(1)- \sum_{s=1}^{t-1} (\bb{R}(s) - \bb{D}(s))\big)]
 \end{array}
\]

The dual objective function is

\[
 \begin{array}{lll}
  g(\bb{y}) &= \sup_{\bb{D}} K(\bb{D},\bb{y}) \\[\smallskipamount]
  &=E[ \alpha \bb{V}(1) \cdot \bb{S}(T) + \sum_{t=1}^{T-1} \bb{V}(1) \cdot \lambda_t + \alpha \sum_{t=1}^{T-1}  \bb{S}(T) \cdot \bb{R}(t) + \sum_{t=1}^{T-1} \bb{v}_t \cdot \bb{b} \\[\smallskipamount]
  & + \sum_{t=1}^{T-1} \bb{w}_t \cdot \big(\bb{m} - \bb{V}(1) -\sum_{s=1}^{t-1} \bb{R}(s) \big) + \sum_{t=1}^{T-1}\sum_{s=1}^{t-1} \lambda_t \cdot \bb{R}(s)] \\[\smallskipamount]
  &+ \sum_{t=1}^{T-1}\sum_{i=1}^{N} \sup_{D_i(t)} \tilde{F}(D_i(t))
 \end{array}
\]
\noindent where $\lambda, \bb{v}, \gamma, \bb{w} \geq 0$ and 

\[
 \tilde{F}(D_i(t))= E[D_i(t) \{S_i(t+1) - \alpha S_i(T) + \gamma_{i,t} - v_{i,t} - \lambda_{i,t}  + \sum_{s=t+1}^{T-1} \big( w_{i,s} -  \lambda_{i,s} \big) \}]
\]

From Lemma 2.1 and 2.2 in Dahl~\cite{Dahl}, it follows that the dual problem is equivalent to
\begin{equation}
\label{eq: dual}
 \begin{array}{lll}
 C + \inf_{\textbf{y} \geq 0} \sum_{t=1}^{T-1}E[\lambda_t \cdot \{\bb{V}(1) + \sum_{s=1}^{t-1} \bb{R}(s)\} \\[\smallskipamount]
  \hspace{3.5cm}+ \bb{v}_t \cdot \bb{b} + \bb{w}_t \cdot \{ \bb{m} - \bb{V}(1) -\sum_{s=1}^{t-1} \bb{R}(s) \}] \\[\smallskipamount]
  \mbox{such that} \\[\smallskipamount]
  \int_A \{ \bb{S}(t+1) - \alpha \bb{S}(T) + \gamma_{t} - \bb{v}_{t} - \sum_{s=t}^{T-1} \lambda_{s} + \sum_{s=t+1}^{T-1} \bb{w}_{s}\}dP =\bb{0} \mbox{ } \forall \mbox{ } A \in \mc{G}_t
 \end{array}
\end{equation}
\noindent where $C:= \alpha E[ \bb{S}(T) \cdot \bb{V}(1) + \sum_{t=1}^{T-1} \bb{S}(T) \cdot \bb{R}(t)]$ and the constraint holds for all $t=1, \ldots, T-1$. The constraint can be rewritten:

\[
 E[S_i(t+1)| \mc{G}_t] - \alpha E[S_i(T)|\mc{G}_t] =  y_{i,t} - \gamma_{i,t} + \lambda_{i,t} +E[\sum_{s=t+1}^{T-1} \big( \lambda_{i,s} -  w_{i,s} \big) | \mc{G}_t], 
\]

\noindent where the constraint holds for $t=1, \ldots, T-1$, $i=1, \ldots, N$, and $y_{i,t}$ and $\gamma_{i,t}$ are $\mc{G}_t$ measurable. In words: The difference between the observed price of electricity and the expected value of the terminal price given the present information is equal to $y_{i,t} - \gamma_{i,t} +\lambda_{i,s} +E[\sum_{s=t+1}^{T-1} \big( \lambda_{i,s} -  w_{i,s} \big)| \mc{G}_t]$, where $\bb{v}, \gamma, \lambda, \bb{w} \geq \bb{0}$.

Note that from the conjugate duality theory, see Rockafellar~\cite{Rockafellar}, the optimal value of the dual problem is an upper bound to the primal maximization problem. So, the optimal value of the hydroelectric dam system is bounded from above by the optimal value of the minimization problem~\eqref{eq: dual}.

In some special cases, the dual problem is simple to solve. Now, assume that $\bb{R}(t) \geq 0$ a.s. for all times $t$, i.e., that there is no natural draining or evaporation from the dams and assume that $\alpha =1$. Also, assume that the electricity price process $\bb{S}(t)$ is a martingale w.r.t. the partial information filtration $(\mc{G}_t)_{t=1}^{T-1}$ and that 

\begin{equation}
\label{eq: assume}
\bb{m} - \bb{V}(1) -\sum_{s=1}^{t-1} \bb{R}(s) \geq \bb{0} \mbox{ a.s.}
\end{equation}

\noindent for all times $t=1, \ldots, T-1$. Note that this final assumption says that almost surely, none of the dams will flood even without draining any water. Since the price process is a martingale, we know that

\[
 E[\bb{S}(t+1) | \mc{G}_t] = E[\bb{S}(T)|\mc{G}_t].
\]

Hence, because of assumption~\eqref{eq: assume}, we see that the optimal solution of the dual problem is to choose $\lambda = \bb{v}=\gamma=\bb{w}=\bb{0}$. Due to the martingale property, this choice implies that the constraints are satisfied and the dual optimal value, $d^*$, is the lowest it can possibly be:

\[
d^*= E[ \bb{S}(T) \cdot \bb{V}(1) + \sum_{t=1}^{T-1} \bb{S}(T) \cdot \bb{R}(t)]. 
\]

If assumption~\eqref{eq: assume} still holds, but the electricity price is a submartingale with respect to the manager's information, it is also easy to see how to find an optimal dual solution. In this case, we know that $ E[\bb{S}(t+1) | \mc{G}_t] \leq E[\bb{S}(T)|\mc{G}_t]$. Therefore, we can let $\lambda = \bb{v} = \bb{w} =0$ and $\gamma_{t} = E[\bb{S}(T)-\bb{S}(t)|\mc{G}_t]$ for all $t=1, \ldots, T-1$. For this choice of $\bb{v} = (\lambda, \bb{v}, \gamma, \bb{w})$, we get the same optimal dual value as in the martingale case.

In the case where the price process is a supermartingale w.r.t. the filtration $(\mc{G}_t)_{t=1}^{T-1}$, it is not obvious what the optimal dual solution will be. In this case, we cannot define $\gamma_{i,t}$ as in the submartingale case, since we must have $\gamma_{i,t} \geq 0$. The same is true if assumption~\eqref{eq: assume} does not hold.       

Note that in the case where assumption~\eqref{eq: assume} holds and the price process is either a martingale or a submartingale, the constraint vectors $\bb{b}$ and $\bb{m}$ (on the maximal production capacities and maximal levels of water in the dams respectively) do not affect the optimal value of the dual problem. The reason for this is as follows: Consider the case where the price process is a martingale. In this case, the manager always expects the current price to be the same as the terminal time price, given her current information. Because of assumption~\eqref{eq: assume}, the manager does not expect to worry about the dams flooding, so she can just wait until the terminal time, and then be left with the remaining water value. This value is $d^*$. Due to the martingale assumption, she does not expect to lose money with this strategy. The argument in the submartingale case is completely parallel. However, since the price process is a submartingale, the manager actually expects to gain money based on this.

\begin{remark}
{\rm As mentioned, we could also consider the situation where the draining of water $\bb{D}(t)$ happens instantaneously, so the electricity can be sold at the known price $\bb{S}(t)$ instead of having to be sold at the next time step for $\bb{S}(t+1)$. However, the calculations in this case become completely identical to those above, except that $\bb{S}(t+1)$ is replaced by $\bb{S}(t)$ throughout. If the price process $(\bb{S}(t))_{t=1}^T$ is $(\mc{G}_t)_{t=1}^T$-adapted, this implies that $E[\bb{S}(t)|\mc{G}_t] = \bb{S}(t)$. Hence, we get a slight simplification of the previous expressions. However, the case where the dual problem is simple to solve is the same as before.}
\end{remark}

\section{Strong duality}
\label{sec: strong_dual}

Another natural question is when strong duality holds, i.e., when is the optimal value of problem~\eqref{eq: problem} equal to the optimal value of the dual problem~\eqref{eq: dual}? In the case where $\Omega$ is finite, the conjugate duality technique reduces to linear programming (LP) duality. Hence, we know from the LP strong duality theorem (see e.g. Vanderbei~\cite{Vanderbei}) that there is no duality gap in this case (since the primal problem clearly has a feasible solution: Just drain whatever flows into the dam). This means that the optimal values of problems~\eqref{eq: problem} and \eqref{eq: dual} are equal in the finite $\Omega$ case. 

We now turn to the case of arbitrary (infinite) $\Omega$, which is more complicated. 

\begin{remark}
 \label{remark: infinite_omega}
 {\rm The case of general $\Omega$ is clearly interesting from a theoretical point of view. It is also relevant for applications. For example, it may be difficult to choose just a few possible future scenarios to study. In this case, considering e.g. a set of scenarios which is normally distributed can be interesting. For instance, such an assumption could reflect that most of the scenarios are somewhere in the middle, but in some cases the realized scenario is very good or very bad.}
\end{remark}

In the remaining part of this section, we make some weak additional assumptions, all of which are very natural:

\begin{assumption}
\label{assumption: bounded}
Assume that:
\begin{itemize}
\item{The price process $\bb{S}(t)$ is bounded on $[0,T]$.}
\item{The inflow process $\bb{R}(t)$ is bounded on $[0,T]$ and $R_i(t) > 0$ for all times $t$ $($so there is always some inflow to the dam$)$.}
\item{The initial water level of the dam is bounded, i.e. $\bb{V}(1) < \infty$.}
\end{itemize}
\end{assumption}

Note that most of the theory on conjugate duality is formulated for convex functions. However, the results can readily be rewritten to the concave case. Since our primal problem~\eqref{eq: problem_max} is a maximization problem, we consider the concave version of the theory. By using this theory, we can prove that there is no duality gap for our problem:

\begin{theorem}
\label{thm: strong_dual}
There is no duality gap for our problem, i.e., the optimal value of problem~\eqref{eq: problem} is equal to the optimal value of problem~\eqref{eq: dual}. Also, there exists a $\bar{y} \in Y$ which solves the dual problem.
\end{theorem}

\begin{proof}
It follows from Example 1 in Pennanen~\cite{Pennanen} and Example 14.29 in Rockafellar and Wets~\cite{RockafellarWets} that our choice of $-F$ is in fact a convex normal integrand, and in particular, it is convex. Hence, the perturbation function $F$ (which we have chosen) is concave.

From Theorem 17 and Theorem 18 a) in Rockafellar~\cite{Rockafellar} rewritten to the concave case, we find that if $F$ is concave and there exists a $\bb{D} \in \mc{D}_{\mc{G}}$ such that the function $\bb{u} \mapsto F(\bb{D}, \bb{u})$ is bounded below on a neighborhood of $0$, the primal and dual optimal values coincide and there exists a $\bar{y} \in Y$ which solves the dual problem.

Since we know that our choice of $F$ is concave, the theorem will follow if we can find $\bb{D} \in \mc{D}_{\mc{G}}$ such that the function $\bb{u} \mapsto F(\bb{D}, \bb{u})$ is bounded below on a neighborhood of $0$.

The only possible problem is the case where $F(\bb{D}, \bb{u}) = -\infty$. This follows because in the other case, $F$ is bounded from below by definition (see equation~\eqref{eq: perturb}) and the assumptions on the stochastic processes involved (see Assumption~\ref{assumption: bounded}).

This means that we have to find a $\{\bb{D}(t)\}_{t \in [0,T]}$ such that for all $\bb{u}$ in a neighborhood of $\bb{0}$, we avoid the case where $F(\bb{D}, \bb{u}) = -\infty$. That means that for our choice of $\bb{D}(t)$, the constraints of equation~\eqref{eq: perturb} have to be satisfied for all such $\bb{u}$. However, this can be achieved by choosing $D_i(t) = \min\{ R_i(t)-\epsilon, V_i(1)-\epsilon, b_i - \epsilon \}$, where epsilon is chosen to be so small that this $D_i(t) > 0$ for all $i=1, \ldots, N, t=1, \ldots, T$. Then, all the constraints of equation~\eqref{eq: perturb} are satisfied for all $\bb{u}$ in an $\epsilon$-neighborhood around $\bb{0}$, and hence $F$ is bounded below on this neighborhood. 

Hence, it follows that there is no duality gap, i.e., that the optimal value of the primal problem~\eqref{eq: problem} is equal to the optimal value of the dual problem~\eqref{eq: dual}.
\end{proof}

\begin{remark}
\label{remark: dual_vs_primal}
 {\rm A weakness of the solving the dual problem instead of the primal problem is that we do not get the optimal drain process directly from the dual. However, in practice when $\Omega$ is finite, one could for instance solve the dual problem by using a dual simplex algorithm. Such an algorithm would provide the optimal draining strategy directly as shadow prices of the dual constrains.} 
\end{remark}

In the case where $\Omega$ is finite, it is typically more efficient to solve the dual problem than it is to solve the primal (from a computational point of view) when the number of constraints in the primal problem is greater than the number of variables. This is the case for the our problem: The number of variables is $N (T-1)|\Omega|$, while the number of constraints is $4 N (T-1)|\Omega|$. Hence, there are $3 N (T-1)|\Omega|$  more constraints than variables. For a large $T$ and $|\Omega|$, this is substantial and it will be faster to use the dual simplex algorithm than using the primal simplex method.

\begin{remark}
{\rm  Note that instead of having the terminal value of water equal to $\alpha \bb{S}(T) \cdot \bb{V}(T)$, we could consider some function $K(\bb{V}(T))$, $K : \mb{R}^N \rightarrow \mb{R}$ of the terminal water level as in Alais et al.~\cite{Alais}. However, since our solution technique eliminates the water level process $\{\bb{V}(t)\}$, we get that
 \[
  K(\bb{V}(T)) = K \big(\sum_{s=1}^{T-1}( \bb{R}(s)- \bb{D}(s)) + \bb{V}(1) \big). 
 \]
In order to be able to use the solution method we have presented, we need to be able to separate out $D_i(t)$ in order to separate the maximizations in order to derive the dual value function. Hence, we need $K$ to be a linear function.}
\end{remark}

\section{A constraint on maximal total production}
\label{sec: max_production}

In this section, we consider the same problem as in Section~\ref{sec: model}, but instead of having constraints on the maximal production capacity of each dam individually, we introduce a constraint on the total production of the whole hydro-dam system. Hence, we can imagine that the system has only one common turbine for all the dams, such as in Figure~\ref{fig: felles_turbin}, instead of having $N$ different turbines as in Figure~\ref{fig: Ndams}.


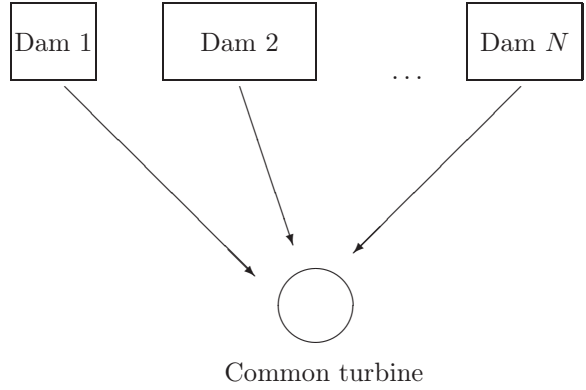
\begin{figure}
\center
\setlength{\unitlength}{1cm}
\begin{picture}(7,5)
\put(0,4){\framebox(1.1,1){Dam $1$}}
\put(2,4){\framebox(2,1){Dam $2$}}
\put(5,4){\mbox{\ldots}}
\put(6,4){\framebox(1.5,1){Dam $N$}}
\put(0.7,3.9){\vector(1,-1){2.5}}
\put(3,3.9){\vector(1,-3){0.7}}
\put(6.7,3.9){\vector(-1,-1){2.2}}
\put(4,1){\circle{1}}
\put(2.8,0){\mbox{Common turbine}}
\end{picture}
\caption{The hydroelectric dam system: Constrained maximal total production}
\label{fig: felles_turbin}
\end{figure}


This means that the question of how to distribute the production between the dams becomes important. This is similar to the problem in Huseby~\cite{Huseby}, however the maximization objective in our paper is different from the one in \cite{Huseby}, and our framework is more general. 

The new problem of the dam manager is:

\begin{equation}
\label{eq: problem_max}
\begin{array}{rlll}
\max_{\{\bb{D(t)}\}} &&E[\sum_{t=1}^{T-1} \bb{D}(t) \cdot \bb{S}(t+1) + \alpha \bb{V}(T) \cdot \bb{S}(T)] \\[\smallskipamount]
\mbox{such that}\\[\smallskipamount]
\bb{V}(t+1) &=& \bb{V}(t) + \bb{R}(t) - \bb{D}(t), \quad t=1, \ldots, T-1 \mbox{ a.s.}, \\[\smallskipamount]
\bb{D}(t)  &\leq& \bb{V}(t) \leq \bb{m}, \quad t=1, \ldots,T-1 \mbox{ a.s.}, \\[\smallskipamount]
0 &\leq& D_i(t), \quad i=1, \ldots, N \mbox{ and } t=1, \ldots, T-1 \mbox{ a.s.} \\[\smallskipamount]
\sum_{i=1}^N D_i(t) &\leq& \tilde{C}, \quad t=1, \ldots, T-1 \mbox{ a.s.}
\end{array}
\end{equation}
\noindent where, like previously, we maximize over all drain processes $\bb{D} \in \mc{D}_{\mc{G}}$, the initial water level $\bb{V}(1)$ is given and $\alpha, \tilde{C}$ are constants. Note that the only difference between problems~\eqref{eq: problem} and \eqref{eq: problem_max} is that the maximum production constraint on each dam is replaced by a maximal total production constraint for the whole facility (the final constraint of problem~\eqref{eq: problem_max}).

Note that there are $(N-1)(T-1)|\Omega|$ more constraints in problem~\eqref{eq: problem_max} than there are in problem~\eqref{eq: problem}, but the two problems have the same number of decision variables. Hence, from the comments after Remark~\ref{remark: dual_vs_primal}, we see that from a computational point of view, using a dual method is less profitable in this case, than for our original problem~\eqref{eq: problem}. However, since there are $N(T-1)|\Omega|$ variables and $(3N + 1)(T-1)|\Omega|$ constraints in \eqref{eq: problem_max}, dual methods are still faster than primal methods.

Problem~\eqref{eq: problem_max} can be rewritten in the same way as problem~\eqref{eq: problem} in Section~\ref{sec: model} by eliminating the water level process. The rewritten problem is

\[
\begin{array}{lll}
\label{eq: problem_max_omskrevet}
\max_{\bb{D}} \hspace{0.05cm} E[\sum_{t=1}^{T-1} \bb{D}(t) \cdot \bb{S}(t+1) + \alpha \bb{S}(T) \cdot \Big(\sum_{s=1}^{T-1} (\bb{R}(s)- \bb{D}(s)) + \bb{V}(1)\Big) ] \\[\smallskipamount]
\mbox{such that}
\end{array}
\]
\begin{equation}
\begin{array}{rlll}
\bb{D}(t)  &\leq& \sum_{s=1}^{t-1} (\bb{R}(s)- \bb{D}(s)) + \bb{V}(1) \leq \bb{m}, \quad t=1, \ldots,T-1 \mbox{ a.s.}, \\[\smallskipamount]
0 &\leq& D_i(t), \quad i=1, \ldots, N \mbox{ and } t=1, \ldots, T-1 \mbox{ a.s.} \\[\smallskipamount]
\sum_{i=1}^N D_i(t) &\leq& \tilde{C},  \quad t=1, \ldots, T-1 \mbox{ a.s.}
\end{array}
\end{equation}
\noindent Like previously, the maximization is over all $\bb{D} \in \mc{D}_{\mc{G}}$.

The dual problem can be derived in just as in Section~\ref{sec: dual}. We omit writing out all details, as the approach is nearly identical to the one already presented.

The Lagrange function is

\[
\begin{array}{lll}
 K(\bb{D},\bb{v}) &=  E[\sum_{t=1}^{T-1} \bb{D}(t) \cdot \bb{S}(t+1) + \alpha \bb{S}(T) \cdot \Big( \sum_{t=1}^{T-1}(\bb{R}(t)- \bb{D}(t)) + \bb{V}(1)\Big)] \\[\smallskipamount]
 & + E[\sum_{t=1}^{T-1} \gamma \cdot \bb{D}(t)] + E[\sum_{t=1}^{T-1} v_{t} (\tilde{C} - +\sum_{i=1}^N D_i(t))] \\[\smallskipamount]
 &+ E[\sum_{t=1}^{T-1} \lambda_{t} \cdot (\bb{V}(1) + \sum_{s=1}^{t-1} (\bb{R}(s) - \bb{D}(s)) - \bb{D}(t))] \\[\smallskipamount]
 & + E[\sum_{t=1}^{T-1} \bb{w}_{t} \cdot \big(\bb{m} - \bb{V}(1)- \sum_{s=1}^{t-1} (\bb{R}(s) - \bb{D}(s))\big)]
 \end{array}
\]

The dual objective function is

\[
 \begin{array}{lll}
  g(\bb{v}) &= \sup_{\bb{D}} K(\bb{D},\bb{v}) \\[\smallskipamount]
  &=E[ \alpha \bb{V}(1) \cdot \bb{S}(T) + \sum_{t=1}^{T-1} \bb{V}(1) \cdot \lambda_t + \alpha \sum_{t=1}^{T-1}  \bb{S}(T) \cdot \bb{R}(t) + \tilde{C}\sum_{t=1}^{T-1} v_t  \\[\smallskipamount]
  & + \sum_{t=1}^{T-1} \bb{w}_t \cdot \big(\bb{m} - \bb{V}(1) -\sum_{s=1}^{t-1} \bb{R}(s) \big) + \sum_{t=1}^{T-1}\sum_{s=1}^{t-1} \lambda_t \cdot \bb{R}(s)] \\[\smallskipamount]
  &+ \sum_{t=1}^{T-1} \sum_{i=1}^{N} \sup_{D_i(t)}  \tilde{F}(D_i(t))
 \end{array}
\]
\noindent where $\lambda, \bb{v}, \gamma, \bb{w} \geq 0$ and 

\[
 \tilde{F}(D_i(t))= E[D_i(t) \{S_i(t+1) - \alpha S_i(T) + \gamma_{i,t} - v_{t} - \lambda_{i,t}  + \sum_{s=t+1}^{T-1} \big( w_{i,s} -  \lambda_{i,s} \big) \}]
\]

From Lemma 2.1 and 2.2 in Dahl~\cite{Dahl}, it follows that the dual problem is equivalent to
\begin{equation}
\label{eq: dual_max}
 \begin{array}{lll}
 C + \inf_{\textbf{y} \geq 0} \sum_{t=1}^{T-1}E[\lambda_t \cdot \{\bb{V}(1) + \sum_{s=1}^{t-1} \bb{R}(s)\} \\[\smallskipamount]
  \hspace{3.5cm} + \tilde{C}\sum_{t=1}^{T-1} v_t + \bb{w}_t \cdot \{ \bb{m} - \bb{V}(1) -\sum_{s=1}^{t-1} \bb{R}(s) \}] \\[\smallskipamount]
  \mbox{such that} \\[\smallskipamount]
  \int_A \{ S_i(t+1) - \alpha S_i(T) + \gamma_{i,t} - v_{t} - \sum_{s=t}^{T-1} \lambda_{s} + \sum_{s=t+1}^{T-1} w_{i,s}\}dP =\bb{0} \mbox{ } \forall \mbox{ } A \in \mc{G}_t, \\[\smallskipamount]
  \hspace{9.4cm} i=1, \ldots, N,
 \end{array}
\end{equation}
\noindent where $C:= \alpha E[ \bb{S}(T) \cdot \bb{V}(1) + \sum_{t=1}^{T-1} \bb{S}(T) \cdot \bb{R}(t)]$ and the constraint holds for all $t=1, \ldots, T-1$. 

Note that for each time $t$, the same variable $v_t$ is in all of the dual constraints (that is, for $i=1, \ldots, N$). This is the main difference between the original dual problem~\eqref{eq: dual} and the dual problem~\eqref{eq: dual_max} for the total maximum production constraint, and provides less flexibility for the dam manager. In the finite scenario space case, where we know that strong duality holds, for $\tilde{C}=\sum_{i=1}^N b_i$, this implies that the optimal value of problem~\eqref{eq: problem_max} for the total production constraint is less than or equal the optimal value for problem~\eqref{eq: problem} for the individual dam constraints. This is what we would expect.

Also, note that in the special case discussed after problem \eqref{eq: dual}, the new dual problem~\eqref{eq: dual_max} is also simple to solve, and the solution is the same as the one found in Section~\ref{sec: dual}. 

Again, we can prove that strong duality holds in the arbitrary $\Omega$ case for problems \eqref{eq: problem_max} and \eqref{eq: dual_max} in the same way as in Section~\ref{sec: strong_dual}. Like previously, in the finite scenario space case, we get strong duality directly from the LP duality  theorem (or as a special case of the strong duality proof of Section~\ref{sec: strong_dual}).

\section{Transfer of water between the dams}
\label{sec: system}

To make our problem more realistic, we add a network structure to the model. To simplify, we consider two dams placed after one another as shown in Figure~\ref{fig: dam}. The first dam lies above the second one. Hence, we can release water from the first dam to the second dam, but not the other way around. Both the dams have turbines with some given maximum capacities. In addition, from the second dam, we can choose to release excess water into the ocean if the dam is about to flood. This cannot be done directly from dam $1$ (only indirectly by first transferring the water to dam $2$) and then releasing it.


\begin{figure}
\center
\setlength{\unitlength}{1cm}
\begin{picture}(8,7)
\put(0,4){\framebox(1.1,1){Dam $1$}}
\put(2.5,3.3){\framebox(2,1){Dam $2$}}
\put(1.3,3.8){\circle{0.5}}
\put(4.7,3.1){\circle{0.5}}
\put(0,3){\mbox{Turbine $1$}}
\put(3.3,2.5){\mbox{Turbine $2$}}
\put(1.5,4.4){\vector(2,-1){0.7}}
\put(1.6,3.5){\vector(1,-1){1.2}}
\put(5,2.8){\vector(1,-1){0.7}}
\put(0,2){\line(1,0){8}}
\put(3.2,1.2){\mbox{Ocean}}
\end{picture}
\caption{The hydroelectric dam system: Transfer of water between dams.}
\label{fig: dam}
\end{figure}
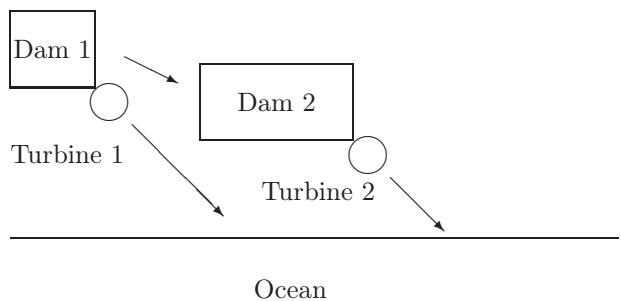


The water from each dam which is released through the respective turbine at time $t$ is, like before, denoted by $D_i(t)$, $i=1,2, t= 1, 2, \ldots, T-1$. The water transferred from dam $1$ to dam $2$ is denoted by $\bar{T}(t)$, $t= 1, 2, \ldots, T-1$ (the T stands for ``transfer''). Let $M_{\bar{T}} > 0$ be a given real number. We require that $0 \leq \bar{T} \leq M_{\bar{T}}$, i.e., one can maximally transfer $M_{\bar{T}}$ units of water from dam $1$ to dam $2$ at each time. 

The amount of water let out from dam $2$ into the ocean at time $t$ is denoted by $O(t)$, $t=1, 2, \ldots, T-1$ (the O stands for ``out''). Let $N_O > 0$ be a real number We require that $0 \leq O(t) \leq N_O$, i.e., that there is a maximum amount of water which can be released at each time. 

We want to study the same problem as in Section~\ref{sec: model}, but adapted to our new system setting. The problem can be written similarly as in Section~\ref{sec: model}, however we need to add the new constraints on $\bar{T}(t), O(t)$ and take into account that the dynamics of the water levels in the dams have changed due to the added network structure:

\begin{equation}
\label{eq: problem_paa_nytt}
\begin{array}{rlll}
\max_{\{\bb{D(t)}, \bar{T}, O\}} &&E[\sum_{t=1}^{T-1} \bb{D}(t) \cdot \bb{S}(t+1) + \alpha \bb{V}(T) \cdot \bb{S}(T)] \\[\smallskipamount]
\mbox{such that}\\[\smallskipamount]
V_1(t+1) &=& V_1(t) + R_1(t) - D_1(t) - \bar{T}(t),\quad t=1, \ldots, T-1 \mbox{ a.s.}, \\[\smallskipamount]
V_2(t+1) &=& V_2(t) + R_2(t) + \bar{T}(t) - D_2(t) -O(t),\quad t=1, \ldots, T-1 \mbox{ a.s.}, \\[\smallskipamount]
\bb{0} &\leq& \bb{D}(t) \leq \bb{b}, \mbox{ } \bb{D}(t)  \leq \bb{V}(t) \leq \bb{m},\quad t=1, \ldots, T-1 \mbox{ a.s.} \\[\smallskipamount]
0 &\leq& \bar{T}(t) \leq M_{\bar{T}}, \mbox{ } 0 \leq O(t) \leq N_0, \quad t=1, \ldots, T-1 \mbox{ a.s.} 
\end{array}
\end{equation}
\noindent Note that the processes $\{\bar{T}(t)\}$ and $\{O(t)\}$ are controls, i.e., they can be chosen by the dam manager. Just as for the drain processes $D_1(t)$ and $D_2(t)$, we assume that $\{\bar{T}(t)\}$ and $\{O(t)\}$ are adapted to the information filtration $\mc{G}$ of the dam manager.

It turns out that we can use the same approach as in Sections~\ref{sec: model}-\ref{sec: dual} to rewrite problem~\eqref{eq: problem_paa_nytt} and then derive the corresponding dual problem. However, note that due to the system structure, we now have

\[
\begin{array}{lll}
 \Delta V_1(t) &=& V_1(t+1) - V_1(t) \\[\smallskipamount]
 &=& R_1(t) - D_1(t) - \bar{T}(t).
\end{array}
\]

Similarly, $\Delta V_2(t) = R_2(t) + \bar{T}(t) - D_2(t) - O(t)$. Hence, by the same kind of calculations as in~\eqref{eq: mellomregning}, we find that

\begin{equation}
\label{eq: V}
\begin{array}{lll}
 V_1(t) &=& V_1(1) + \sum_{t=1}^{T-1} \{R_1(t) - \bar{T}(t) - D_2(t)\} \\[\smallskipamount]
 V_2(t) &=& V_2(1) + \sum_{t=1}^{T-1} \{R_2(t) + \bar{T}(t) - D_2(t)- O(t)\}.
\end{array}
\end{equation}

By proceeding precisely as in Section~\ref{sec: dual}, we can derive the corresponding perturbation function, Lagrange function and dual objective function. We omit writing out the perturbation function as it is lengthy and very similar to the one in Section~\ref{sec: dual}. The Lagrange function is:

\[
\begin{array}{lll}
 K(\bb{D},\bar{T}, O, \bb{y}) &=  E[\sum_{t=1}^{T-1} \bb{D}(t) \cdot \bb{S}(t+1) + \alpha \bb{S}(T) \cdot \bb{V}(T)]\\[\smallskipamount]
 & + E[\sum_{t=1}^{T-1} \bb{v}_{t} \cdot (\bb{b} - \bb{D}(t))] + E[\sum_{t=1}^{T-1} \lambda_{t} \cdot (\bb{V}(t) - \bb{D}(t))] \\[\smallskipamount]
 & + E[\sum_{t=1}^{T-1} \bb{w}_{t} \cdot \big(\bb{m} - \bb{V}(t)\big)]
 \end{array}
\]
\noindent where $\bb{V}(t)$ should be replaced with the expressions in equation~\eqref{eq: V} to get a Lagrange function only depending on the control processes $\bb{D}(t), \bar{T}(t), O(t)$, the input processes $\bb{R}(t), \bb{S}(t)$ and the initial water levels in the dams $\bb{V}(1)$.

The dual objective function is 

\[
 \begin{array}{lll}
  g(\bb{y}) &= \sup_{\bb{D}} K(\bb{D},\bb{y}) \\[\smallskipamount]
  &=E[ \alpha \bb{V}(1) \cdot \bb{S}(T)] + \sum_{t=1}^{T-1} E[ \bb{V}(1) \cdot \lambda_t + \alpha   \bb{S}(T) \cdot \bb{R}(t) +  \bb{v}_t \cdot \bb{b} \\[\smallskipamount]
  & + v_{\bar{T}}(t)M_{\bar{T}} + v_{O}(t)N_{O} + \bb{w}_t \cdot \big(\bb{m} - \bb{V}(1) -\sum_{s=1}^{t-1} \bb{R}(s) \big) + \sum_{s=1}^{t-1} \lambda_t \cdot \bb{R}(s) ] \\[\smallskipamount]
  &+ \sum_{t=1}^{T-1} \{ \sup_{\bar{T}(t)} \bar{K}(\bar{T}(t)) +  \sup_{O(t)} G(O(t)) + \sum_{i=1}^{N}  \sup_{D_i(t)} \tilde{F}(D_i(t))  \} 
 \end{array}
\]
\noindent where $\lambda, \bb{v}, \gamma, \bb{w} \geq 0$ and for $i=1,2$,

\[
\begin{array}{lll}
 \tilde{F}(D_i(t)) &=& E[D_i(t) \{S_i(t+1) - \alpha S_i(T) + \gamma_{i,t} - v_{i,t} - \lambda_{i,t}  + \sum_{s=t+1}^{T-1} \big( w_{i,s} -  \lambda_{i,s} \big) \}],\\[\smallskipamount]
 \bar{K}(\bar{T}(t)) &=& \bar{T}(t) \{ \alpha(S_2(T) - S_1(T)) + \sum_{s=1}^{t-1} [\lambda_s^{(2)} - \lambda_s^{(1)}] + \sum_{s=1}^{t-1} [w_s^{(2)} - w_s^{(1)}] \\[\smallskipamount]
 &&+ \gamma_{\bar{T}}(t) + v_{\bar{T}}(t)\}, \\[\smallskipamount]
G(O(t)) &=& O(t) \{-\alpha S_2(T) + \sum_{s=1}^{t-1} w_s^{(2)}- \sum_{s=1}^{t-1} \lambda_s^{(2)} + \gamma_O(t) - v_O(t)  \}. 
\end{array}
 \]

From Lemma 2.1 and 2.2 in Dahl~\cite{Dahl}, it follows that the dual problem is equivalent to
\begin{equation}
\label{eq: dual_system}
 \begin{array}{lll}
 C + \inf_{\textbf{y} \geq 0} \sum_{t=1}^{T-1}E[\lambda_t \cdot \{\bb{V}(1) + \sum_{s=1}^{t-1} \bb{R}(s)\} + v_{\bar{T}}(t)M_{\bar{T}} + v_{O}(t)N_{O} \\[\smallskipamount]
  \hspace{3.5cm}+ \bb{v}_t \cdot \bb{b} + \bb{w}_t \cdot \{ \bb{m} - \bb{V}(1) -\sum_{s=1}^{t-1} \bb{R}(s) \}] \\[\smallskipamount]
  \mbox{such that} \\[\smallskipamount]
  \int_A \{ \bb{S}(t+1) - \alpha \bb{S}(T) + \gamma_{t} - \bb{v}_{t} - \sum_{s=t}^{T-1} \lambda_{s} + \sum_{s=t+1}^{T-1} \bb{w}_{s}\}dP =\bb{0} \mbox{ } \forall \mbox{ } A \in \mc{G}_t, \\[\smallskipamount]
  \int_A \{\alpha(S_2(T) - S_1(T)) + \sum_{s=1}^{t-1} [\lambda_s^{(2)} - \lambda_s^{(1)}] + \sum_{s=1}^{t-1} [w_s^{(2)} - w_s^{(1)}] \\[\smallskipamount]
  \hspace{5cm} + \gamma_{\bar{T}}(t) + v_{\bar{T}}(t) \} dP =0 \mbox{ for all }  A \in \mc{G}_t, \\[\smallskipamount]
  \int_A \{-\alpha S_2(T) + \sum_{s=1}^{t-1} w_s^{(2)}- \sum_{s=1}^{t-1} \lambda_s^{(2)} + \gamma_O(t) - v_O(t) \} dP = 0 \mbox{ for all } A \in \mc{G}_t. 
 \end{array}
\end{equation}
\noindent where $C:= \alpha E[ \bb{S}(T) \cdot \bb{V}(1) + \sum_{t=1}^{T-1} \bb{S}(T) \cdot \bb{R}(t)]$ and the constraints hold for all $t=1, \ldots, T-1$.

Note that the duality approach still works precisely as in Section~\ref{sec: dual}, despite the added complexity of the structure between the dams. The same would be true for alternative system structure between the dams. The general framework of conjugate duality allows up to upper bound our problem for all such structures. However, if the system structure is very complex, the dual problem will also be more complex. In problem~\eqref{eq: dual_system}, we see that the added complexity leads to two extra constraints (corresponding essentially to the two added control variables). In addition, the dual objective function has two extra terms added compared to the dual problem \eqref{eq: dual} of Section~\ref{sec: dual}. However, these are of a very simple form.

Note also that in problem~\eqref{eq: dual}, the constraints are fairly simple because constraint $i$ only depends on $S_i(t)$. This means that the constraints are all separate. This is natural, as the dams are also assumed to be detached from one another. In problem~\eqref{eq: dual_system}, because transferring water between the dams is possible, we see that this is reflected by the fact that the dual constraints are connected (f.ex. both $S_1(T)$ and $S_2(T)$ are a part of the second constraint in equation~\eqref{eq: dual_system}).

Like before, the dual problem is simple to solve in some special cases. We assume that $\bb{R}(t) \geq 0$ for all times $t$ (i.e., there is no natural draining or evaporation from the dams) and assume that $\alpha =1$. Also, assume that the electricity price process $\bb{S}(t)$ is a martingale w.r.t. the partial information filtration $(\mc{G}_t)_{t=1}^{T-1}$ and that the dams' technologies are identical, so $S_1(t)=S_2(t)$ for all $t=1, \ldots, T-1$. Finally, assume that

\begin{equation}
\label{eq: assume_system}
\bb{m} - \bb{V}(1) -\sum_{s=1}^{t-1} \bb{R}(s)  \geq \bb{0} \mbox{ a.s.}
\end{equation}

\noindent for all times $t=1, \ldots, T-1$ (i.e., almost surely, none of the dams will flood even if we do not drain any water). Since the price process is a martingale and because of assumption~\eqref{eq: assume_system}, we see that the optimal solution of the dual problem is to choose $\lambda = \bb{v}=\gamma=\bb{w}=\bb{0}$. Due to the martingale property, this choice implies that the constraints are satisfied and the dual optimal value, $d^*$, is the lowest it can possibly be:

\[
d^*= E[ \bb{S}(T) \cdot \bb{V}(1) + \sum_{t=1}^{T-1} \bb{S}(T) \cdot \bb{R}(t)]. 
\]

Note that in this case there is no difference in the optimal solution of problem~\eqref{eq: dual} and problem~\eqref{eq: dual_system}. However, this is quite natural due to the strict assumptions made in order to derive this solution. In a completely corresponding way as in Section~\ref{sec: strong_dual}, we can prove that strong duality holds for this modified setting. Hence, the optimal value of problem~\eqref{eq: dual_system} is equal to the optimal value of the original problem~\eqref{eq: problem_paa_nytt}. 


\begin{remark}
{\rm As an alternative to the duality method, we could try to use a more direct approach and solve the primal problem~\eqref{eq: problem_paa_nytt} directly. The natural idea would be to use dynamic programming. However, the constraints in problem~\eqref{eq: problem} complicate this method significantly. According to Dohrman and Robinett~\cite{Dohrman}, even in the deterministic case, inequality constraints such as the those in equation~\eqref{eq: problem}, demand more sophisticated methods than unconstrained or equality constrained problems. One complicating factor is to determine which of the constraints are binding (i.e., hold with equality) in the optimum.

In Dahl and Stokkereit~\cite{DahlStokkereit}, a method combining Lagrange duality and some method of stochastic control, for instance dynamic programming, is derived. However, this approach is based on having equality constraints, and the proofs of that paper no longer work when considering inequality constraints instead. As already mentioned, in the case where the scenario space $\Omega$ is assumed to be finite, the primal problem~\eqref{eq: problem} is a linear programming problem which can be solved efficiently by well-known methods (see for example Vanderbei~\cite{Vanderbei}). Hence, the difficulty is to handle the case where $\Omega$ is not finite. 

In the deterministic case, there are ways to overcome this problem and find the optimal control under the constraints; active set theory, projected Newton methods or interior point methods, see Dohrman and Robinett~\cite{Dohrman}. In particular, for ``box''-type constraints, projected Newton algorithms have been shown to be efficient. However, to the best of our knowledge, such algorithms have not been generalized to the stochastic setting. }



\end{remark}

\appendix

\section{Conjugate duality and paired spaces}
\label{sec: conjugate}

This appendix is almost the same as the one in Dahl~\cite{Dahl}, and is included for the reader's convenience.

Conjugate duality theory (also called convex duality), introduced by Rockafellar~\cite{Rockafellar}, provides a method for solving very general optimization problems via dual problems. The following theory is, as is common in optimization literature, formulated for minimization problems. However, it can easily be translated to a maximization context by using that $\min f(x) = - \max -f(x)$.

Let $X$ be a linear space, and let $f: X \rightarrow \mathbb{R}$ be a function. The minimization problem $\min_{x \in X} f(x)$ is called the \emph{primal problem}, denoted $(P)$. In order to apply the conjugate duality method to the primal problem, we consider an abstract optimization problem $\min_{x \in X} F(x,u)$ where $F: X \times U \rightarrow \mathbb{R}$\index{$F(x,u)$} is a function such that $F(x,0) = f(x)$, $U$ is a linear space and $u \in U$ is a parameter chosen depending on the particular problem at hand. The function $F$ is called the \emph{perturbation function}. We would like to choose $(F,U)$ such that $F$ is a closed, jointly convex function of $x$ and $u$.

Corresponding to this problem, one defines the \emph{optimal value function}
\begin{equation}
 \varphi(u)\index{$\varphi(\cdot)$} := \inf_{x \in X} F(x,u) \hspace{0,05cm}, \hspace{0,3cm} u \in U.
\end{equation}
Note that if the perturbation function $F$ is jointly convex, then the optimal value function $\varphi(\cdot)$ is convex as well.

A \emph{pairing} of two linear spaces $X$ and $V$ is a real-valued bilinear form $\langle \cdot , \cdot \rangle$ on $X \times V$. Assume there is a pairing between the spaces $X$ and $V$. A topology on $X$ is \emph{compatible} with the pairing if it is a locally convex topology such that the linear function $\langle \cdot , v \rangle$ is continuous, and any continuous linear function on $X$ can be written in this form for some $v \in V$. A compatible topology on $V$ is defined similarly. The spaces $X$ and $V$ are \emph{paired spaces} if there is a pairing between $X$ and $V$ and the two spaces have compatible topologies with respect to the pairing. An example is the spaces $X = L^{p}(\Omega, F, P)$ and $V = L^{q}(\Omega, F, P)$, where $\frac{1}{p} + \frac{1}{q} = 1$. These spaces are paired via the bilinear form $\langle x , v \rangle = \int_{\Omega} x(s) v(s) dP(s)$.

In the following, let $X$ be paired with another linear space $V$, and $U$ paired with the linear space $Y$. The choice of pairings may be important in applications.  Define the \emph{Lagrange function} $K: X \times Y \rightarrow \bar{\mathbb{R}}$ to be $K(x,y) := \inf\{F(x,u) + \langle u,y \rangle : u \in U\}$. The following Theorem~\ref{thm: Lagrange} is from Rockafellar~\cite{Rockafellar} (see Theorem 6 in \cite{Rockafellar}).
\begin{theorem}
 \label{thm: Lagrange}
The Lagrange function $K$ is closed, concave in $y \in Y$ for each $x \in X$, and if $F(x,u)$ is closed and convex in $u$
\begin{equation}
 \label{eq:box}
f(x) = \sup_{y \in Y} K(x,y).
\end{equation}
\end{theorem}
For the proof of this theorem, see Rockafellar~\cite{Rockafellar}. Motivated by Theorem~\ref{thm: Lagrange}, we define the \emph{dual problem} of $(P)$,
\begin{eqnarray}
 (D) \hspace{1cm} \max_{y \in Y} g(y) \nonumber
\end{eqnarray}
\noindent where $g(y) := \inf_{x \in X} K(x,y)$.

One reason why problem $(D)$ is called the dual of the primal problem $(P)$ is that, from equation (\ref{eq:box}), problem $(D)$ gives a lower bound on problem $(P)$. This is called \emph{weak duality}. Sometimes, one can prove that the primal and dual problems have the same optimal value. If this is the case, we say that there is \emph{no duality gap} and that \emph{strong duality holds}. The next theorem (see Theorem 7 in Rockafellar~\cite{Rockafellar}) is important:
\begin{theorem}
 \label{thm: dualprob}
The function $g$ in $(D)$ is closed and concave. Also
\begin{eqnarray}
\sup_{y \in Y} g(y) = \cl (\co (\varphi))(0)  \nonumber
\end{eqnarray}
and
\begin{eqnarray}
\inf_{x \in X} f(x) = \varphi(0). \nonumber
\end{eqnarray}
\end{theorem}
\noindent (where $\cl$ and $\co$ denote respectively the closure and the convex hull of a function, see Rockafellar~\cite{RockafellarConvex}). For the proof, see Rockafellar~\cite{Rockafellar}. Theorem~\ref{thm: dualprob} implies that \emph{if the value function $\varphi$ is convex, the lower semi-continuity of $\varphi$ is a sufficient condition for the absence of a duality gap}.

\end{document}